\documentclass{amsart}[12pt]
\usepackage[dvipdfm]{graphicx}
\usepackage{amsmath,amsthm,amssymb}
\usepackage{}

\makeatletter
\newtheorem*{rep@theorem}{\rep@title}
\newcommand{\newreptheorem}[2]{%
\newenvironment{rep#1}[1]{%
 \def\rep@title{#2 \ref{##1}}%
 \begin{rep@theorem}}%
 {\end{rep@theorem}}}
\makeatother

\newtheorem{prop}{\bf Proposition}[section]
\newtheorem{thm}[prop]{\bf Theorem}
\newreptheorem{theorem}{Theorem}
\newtheorem{lem}[prop]{\bf Lemma}
\newtheorem{cor}[prop]{\bf Corollary}
\newtheorem{claim}[prop]{\bf Claim}
\theoremstyle{definition}
\newtheorem{defi}[prop]{Definition}
\newtheorem{rmk}[prop]{Remark}
\newtheorem{ex}[prop]{Example}

\makeatletter
\@addtoreset{equation}{section}

\makeatother

\newcommand{\PSL}{\mathrm{PSL}(2,\mathbb{C})}

\newcommand{\map}{\rightarrow}

\title{On commensurability of fibrations on a hyperbolic 3-manifold}
\author{Hidetoshi Masai}
\address{Department of Mathematical and Computing Sciences, Tokyo Institute of
Technology, O-okayama, Meguro-ku, Tokyo 152-8552 Japan}
\email{masai9 at is.titech.ac.jp}
\date{}
\subjclass[2000]{Primary~57M25. Secondary~57N10}

\begin{document}

\begin{abstract}
We discuss fibered commensurability of fibrations on a hyperbolic 3-manifold,
a notion introduced by Calegari, Sun and Wang.
We construct manifolds with non-symmetric but commensurable fibrations on the same fibered face.
We also prove that if a given manifold $M$ does not have any hidden symmetries, then $M$ does not admit
non-symmetric but commensurable fibrations.
Finally, Theorem 3.1 of Calegari, Sun and Wang shows that every hyperbolic fibered commensurability class contains a unique minimal element.
In this paper we provide a detailed discussion on the proof of the theorem in the cusped case.
\end{abstract}
\maketitle

\section{Introduction}
In this paper, we are mainly interested in fibered hyperbolic 3-manifolds with the first Betti number greater than or equal to $2$.
Thurston \cite{Thu.norm} showed that such a manifold admits infinitely many distinct fibrations (see also \S \ref{cms.mfd}).
It is an interesting question to investigate the relationship between such fibrations.

In \cite{CSW}, Calegari, Sun and Wang defined the notion of fibered commensurability, 
which gives rise to an equivalence relation on fibrations.
An {\em automorphism} on a surface is the isotopy class of a self-homeomorphism of the surface.
For any fibration on a 3-manifold, we have the pair $(F,\phi)$ of the fiber surface $F$, 
and the monodromy automorphism $\phi$.
Since the monodromy is determined up to conjugacy in the mapping class group of $F$, we use the notation $(F,\phi)$ to denote the conjugacy class.
Then commensurability of fibrations is defined as follows.
\begin{defi}[\cite{CSW}]\label{def.cover}
A pair $(\widetilde{F},\widetilde{\phi})$ {\em covers} $(F,\phi)$ if there is a finite cover $\pi:\widetilde{F}\map F$ and representative homeomorphisms $\tilde{f}$ of $\widetilde\phi$ and $f$ of
$\phi$ so that $\pi \widetilde{f} = f \pi$ as maps $\widetilde{F}\map F$.
\end{defi}
\begin{defi}[\cite{CSW}]\label{def.cms}
Two pairs $(F_1, \phi_1)$ and $(F_2, \phi_2)$ are {\em commensurable} if there is a surface $\widetilde{F}$,
automorphisms $\widetilde{\phi_1}$ and $\widetilde{\phi_2}$, and nonzero integers $k_1$ and $k_2$, so that $(\widetilde{F},\widetilde\phi_i)$ covers
$(F_i,\phi_i)$ for $i = 1,2 $ and if $\widetilde{\phi_1}^{k_1} = \widetilde{\phi_2}^{k_2}$ as automorphisms of $\widetilde{F}$.
\end{defi}
%
For the remainder of the paper, we consider fibrations on hyperbolic 3-manifolds.
In this case, the monodromy of each fibration is always {\em pseudo-Anosov} (see \S \ref{sec.PA} for the definition).
The {\em normalized entropy} of a conjugacy class $(F,\phi)$ is defined as $\chi(F)\log(\lambda(\phi))$, 
where $\chi(F)$ is the Euler characteristic of $F$ and $\lambda(\phi)$ is the dilatation of $\phi$.
In \S \ref{sec.defi}, we observe that the normalized entropies of commensurable fibrations on 
the same hyperbolic 3-manifold agree.
Then we demonstrate an example of a manifold such that two of its fibrations are commensurable if and only if 
they share the same normalized entropy.
We also give an example of a manifold with two non-commensurable fibrations of the same normalized entropy.

In this paper, we study commensurable fibrations on a hyperbolic 3-manifold in the context of a fibered face.
A {\em fibered face} is a face of the Thurston norm ball whose rational points correspond to fibrations of the 3-manifold 
 and a {\em fibered cone} is a cone over a fibered face (see \S \ref{sec.entropy} for details).
Two fibrations on $M$ are said to be {\em symmetric} if there exists a self-homeomorphism $\varphi:M\map M$ that
maps one to the other.
In \cite[Remark 3.9]{CSW}, Calegari, Sun, and Wang asked if there is an example of two fibrations on the same closed manifold, which are commensurable but have fibers distinguished by their genera.
The following theorem provides such a construction in the cusped case.
In this theorem fibers are distinguished by their Euler characteristics (see \S\ref{cms.mfd} for a proof).

\begin{thm}\label{thm.genus}
There are hyperbolic 3-manifolds with non-symmetric but non-commensurable fibrations whose corresponding elements in $H^1(M;\mathbb{Z})$
are in the same fibered cone.
\end{thm}

On the other hand, if $M$ has no hidden symmetries, then such fibrations do not exist.
Here, a (finite-volume) hyperbolic 3-manifold $M = \mathbb{H}^3/\Gamma$ is said to have {\em hidden symmetries} if
$[C^+(\Gamma):N^+(\Gamma)]>1$, where $C^+ (\Gamma)$ (resp. $N^+ (\Gamma)$) is the 
commensurator (resp. normalizer) of $\Gamma$, see \S \ref{cms.mfd} for details.
\begin{thm}\label{thm.hidden}
Suppose that $M$ is a hyperbolic $3$-manifold that does not have hidden symmetries.
Then, any pair of fibrations of $M$ is either symmetric or non-commensurable, but not both.
\end{thm}

Theorem \ref{thm.genus} and Theorem \ref{thm.hidden} are motivated by the fact that up to isotopy,
there are only finitely many commensurable fibrations on a hyperbolic 3-manifold.
This fact is a corollary of the following.
\begin{thm}[see also Theorem 3.1 of \cite{CSW}]\label{csw.Hyperbolic}
Every commensurability class of hyperbolic fibered pairs contains a unique (orbifold) minimal element.
\end{thm}
Here the notion of a fibered pair is a generalization of the notion of a pair $(F,\phi)$, see \S \ref{sec.defi} for details.
The proof in \cite{CSW} works for the closed case.
In \S \ref{sec.defi} we extend it to the case where the manifolds have boundary (Theorem \ref{Hyperbolic}).
Further, as a corollary of this extension, we show examples of manifolds such that every fibration is the minimal element
in its commensurability class (Corollary \ref{HKT}).

Commensurability classes are defined using the transitive hull of the relation in Definition 1.2.
In \S \ref{sec.defi} we also discuss the transitivity of commensurability.
We show that if the automorphisms are pseudo-Anosov (that is to say, in the hyperbolic case),
then commensurability is transitive.

{\bf Acknowledgment.}
I would like to thank Sadayoshi Kojima and members of his research group for helpful conversations and advice.
I would also like to thank Eriko Hironaka for her inspiring series of lectures about pseudo-Anosov maps at Tokyo Tech and for telling me the argument in Example \ref{Hironaka}.
I would also like to thank Neil Hoffman, Tamas Kalman 
and Hongbin Sun for catching some errors in an earlier version of this paper.
Finally I would like to thank the anonymous referee for helpful suggestions. 
This work was partially supported by JSPS Research Fellowship for Young Scientists.

\section{Preliminaries}\label{sec.defi}
In this section, we recall the definitions and basic facts about commensurability of fibrations.
Most of the contents in this section are discussed in \cite{CSW}.
In this paper, unless otherwise stated, by a {\em surface} and a {\em hyperbolic 3-manifold}, we mean a compact connected orientable 2-manifold possibly with boundary and of negative Euler characteristic, and a connected, orientable, complete hyperbolic 3-manifold of finite volume respectively.

\subsection{Fibered pairs}\label{sec.fcms}
Given homeomorphism $f:F\map F$, there is an associated 3-manifold $M$ called the {\em mapping torus} of $f$;
$$ M = F\times [0,1]/((f(x),0)\sim (x,1)).$$
Since the mapping tori of conjugate automorphisms are homeomorphic to each other,
each conjugacy class $(F,\phi)$ determines a homeomorphism class of 3-manifolds by taking mapping tori.
We use the notation $[F,\phi]$ to denote the mapping torus associated to $(F,\phi)$.
$F$ is called the {\em fiber} and $\phi$ is called the {\em monodromy} of $[F,\phi]$.

In this paper, we focus on fibrations of a fixed hyperbolic 3-manifold $M$.
Each fibration on $M$ over the circle determines an element  of $H^1(M;\mathbb{Z})$ and if $\omega\in H^1(M;\mathbb{Z})$ corresponds to a fibration, then there is an associated pair $(F,\phi)$ of $F$ the fiber and $\phi$ the monodromy (i.e. $[F,\phi]$ is homeomorphic to $M$).
Note that this correspondence of $\omega$ and $(F,\phi)$ is well-defined up to the conjugation of $(F,\phi)$.

In \S \ref{hyperbolic}, we discuss Theorem 3.1 of \cite{CSW} for the case of fibered manifolds with boundary.
To state the theorem it is convenient to define a {\em fibered pair} which is a generalization of a pair of type $(F,\phi)$.
We also enlarge our attention to orbifolds.
An $n$-orbifold is a space that is locally modeled on a quotient of an open ball in $\mathbb{R}^n$ by a finite group.
See \cite{W} and Chapter 13 of \cite{Thur} for more details.
\begin{defi}[\cite{CSW}]
A {\em fibered pair} is a pair $(M,\mathcal{F})$ where $M$ is a compact 3-manifold with boundary a union of tori and Klein bottles, 
and $\mathcal{F}$ is a foliation by compact surfaces.
More generally, an orbifold fibered pair is a pair $(O,\mathcal{G})$ where $O$ is a compact 3-orbifold, and $\mathcal{G}$ is a foliation of $O$ by compact $2$-orbifolds.
\end{defi}
\begin{defi}[\cite{CSW}]
A fibered pair $(\widetilde{M}, \widetilde{\mathcal{F}})$ {\em covers} $(M,\mathcal{F})$
if there is a finite covering of manifolds 
$\pi:\widetilde{M}\map M$ such that $\pi^{-1}(\mathcal{F})$ is isotopic to $\widetilde{\mathcal{F}}$.
Two fibered pairs $(M_1,\mathcal{F}_1)$ and $(M_2, \mathcal{F}_2)$ are {\em (fibered) commensurable} if there is a third fibered pair $(\widetilde{M}, \widetilde{\mathcal{F}})$
that covers both.
\end{defi}

For a given pair $(F,\phi)$, the mapping torus 
$[F,\phi]$ has a foliation $\mathcal{F}$ by surface leaves, which are homeomorphic to $F$ and hence
there is a corresponding fibered pair $([F,\phi], \mathcal{F})$.

Unlike the case of commensurability in Definition \ref{def.cms}, 
it is easy to see that commensurability of fibered pairs is transitive. 
Suppose $(M_i,\mathcal{F}_i)$ and $(M_{i+1},\mathcal{F}_{i+1})$ are commensurable for $i = 1,2$
and $(\widetilde{M}_{12},\widetilde{\mathcal{F}}_{12})$ (resp.  $(\widetilde{M}_{23},\widetilde{\mathcal{F}}_{23})$)
is a common covering pair of $(M_1,\mathcal{F}_1)$ and $(M_{2},\mathcal{F}_{2})$ (resp. $(M_2,\mathcal{F}_2)$ and $(M_{3},\mathcal{F}_{3})$).
Then there is a covering $p:\widetilde{N}\map M_2$ that corresponds to 
$p_{*}^{12}\pi_1(\widetilde{M_{12}})\cap p_{*}^{23}\pi_1(\widetilde{M_{23}})<\pi_1(M_2)$ 
where $p^{12}:\widetilde{M_{12}}\map M_2$ and $p^{23}:\widetilde{M_{23}}\map M_2$ are the covering maps.
Then $(\widetilde{N},p^{-1}(\mathcal{F}_2))$ covers both $(M_1,\mathcal{F}_1)$ and $(M_3,\mathcal{F}_3)$.
Thus we see that fibered commensurability is a transitive relation.

We define another equivalence relation on fibered pairs so that the covering relation will be a partial order.

\begin{defi}[\cite{CSW}]
We say that two fibered pairs $(M,\mathcal{F})$ and $(N,\mathcal{G})$ are {\em covering equivalent} if each covers the other. 
We call a covering equivalence class {\em minimal} if no representative covers any element of another covering equivalence class.
\end{defi}

\begin{rmk}[see also Remark 2.9 of \cite{CSW}]
Each covering equivalence class of the fibered pair associated to $(F,\phi)$
contains exactly one fibered pair unless $\phi$ is periodic.
Therefore, when we consider pseudo-Anosov automorphisms,
by abusing notation, we use the word ``element" for each covering equivalent class.
\end{rmk}

\subsection{Pseudo Anosov automorphisms}\label{sec.PA}
The automorphisms on a compact surface are classified into 3 types;
periodic, reducible, and {\em pseudo-Anosov} (see \cite{Thu}, \cite{CB}).
By a result of Thurston, it turns out that the (interior of) mapping torus $[F,\phi]$ admits hyperbolic metric of finite volume
if and only if
the automorphism $\phi$ is pseudo-Anosov (see \cite{Thu}, cf. \cite{OK}).
\begin{defi}
A homeomorphism $f:F\map F$ is a {\em pseudo-Anosov homeomorphism} if there is a pair of transverse measured singular foliations
$(\mathcal{F}^s,\mu^s)$ and $(\mathcal{F}^u,\mu^u)$ on $F$ 
and a positive real number $\lambda$
so that $f(\mathcal{F}^u) = \mathcal{F}^u, f(\mu^u) = \lambda \mu^u$ and $f(\mathcal{F}^s) = \mathcal{F}^s, f(\mu^s) = (1/\lambda) \mu^s$.
$(\mathcal{F}^u,\mu^u)$ ( resp. $(\mathcal{F}^s,\mu^s)$) is called the unstable (resp. stable) measured singular foliation associated to $f$.
\end{defi}
See Figure \ref{pA} for a shape of the singularities of $(\mathcal{F}^s,\mu^s)$ and $(\mathcal{F}^u,\mu^u)$.
An automorphism $\phi$ is said to be {\em pseudo-Anosov} if it has a pseudo-Anosov homeomorphism as a representative.
We call the positive real number $\lambda$ the {\em dilatation} of pseudo-Anosov automorphism $\phi$ and denote it by $\lambda(\phi)$.

In some cases, it is convenient to consider the restriction of automorphisms on  the interior ${{\rm Int}(F)}$ of $F$.
By considering $\phi\vert_{{\rm Int}(F)}$, we get a pseudo-Anosov automorphism on ${\rm Int}(F)$ and
by abusing the notation we also denote it by $\phi$.
Note that  ${\rm Int}(F)$ can be regarded as a surface with finitely many punctures,
each corresponding to a boundary component of $F$.
Then the singularities of $\mathcal{F}^s$ and $\mathcal{F}^u$ lie on ${\rm Int}(F)$ or the punctures .
We denote the set of points and punctures that correspond to the singular points of associated singular foliations by Sing($\phi$).
\begin{figure}[h]
\includegraphics[bb = 0 0 620 618, scale=0.2]{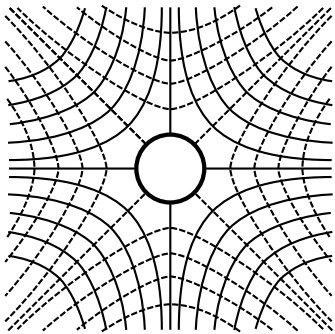}
\caption{A shape of a singularity of degree $4$ at the boundary.}
\label{pA}
\end{figure}

\subsection{Uniqueness of the minimal element}\label{hyperbolic}
In this subsection, we give a detailed discussion of Theorem \ref{csw.Hyperbolic} for the case where manifolds have boundary.
By passing to a finite covering we may assume $\mathcal{F}$ to be co-orientable and hence $M$ fibers over the circle i.e.
$M$ is the mapping torus $[F,\phi]$ of some surface $F$ and pseudo-Anosov map $\phi$.
Since we are dealing with commensurability classes, it suffices to discuss the case where the foliations are co-orientable.
The proof in \cite{CSW} assumes that all singular points of the singular foliations associated to $\phi$ lie on the interior of $F$.
We prove this result for the case where some of the singular points lie on the boundary.
This corresponds to the case where Sing($\phi$) contains some punctures, 
by restricting the automorphism on  the interior ${\rm Int}(F)$ of $F$.

\begin{thm}[see also \cite{CSW}]\label{Hyperbolic}
Let $(M,\mathcal{F})$ be a hyperbolic co-orientable fibered pair and $(F,\phi)$ be the pair associated to $(M,\mathcal{F})$.
Then the commensurability class of $(M,\mathcal{F})$ contains a unique minimal (orbifold) element.
Moreover, if ${\rm Int}(F)\cap {\rm Sing}(\phi) = \emptyset$, then the minimal element is a manifold.
\end{thm}
\begin{proof}
The first part of this proof was proved in \cite{CSW}.
We first recall the argument in \cite{CSW} which we will use here.
The stable and unstable singular foliations $\mathcal{F}^s$ and $\mathcal{F}^u$ associated to $\phi$ determine a unique singular Sol metric on ${\rm Int}(M)$.
Pulling back this metric to the universal cover $\pi:\widetilde{M}\rightarrow{\rm Int}(M)$, $\widetilde{M}$ becomes a simply connected singular Sol manifold.
Each fiber of $\widetilde{M}$ is a singular Euclidean plane. 
Let $\Lambda$ be the full isometry group of the singular Sol metric.
By appealing to the local Sol metric of $\widetilde{M}$, it can be verified that each element of $\Lambda$ preserves the foliation by the singular Euclidean planes.
Since $\pi_1(M)<\Lambda$, we have the covering 
$(\widetilde{M}, \widetilde{\mathcal{F}})/\pi_1(M)\map (\widetilde{M}, \widetilde{\mathcal{F}})/\Lambda$.
Then we see that for any pair $(M',\mathcal{F}')$ commensurable with $(M,F)$ the group $\pi_1(M')$ embeds into $\Lambda$  and hence $(M',\mathcal{F}')$ covers $(\widetilde{M}, \widetilde{\mathcal{F}})/\Lambda$.
Therefore to prove the theorem, it suffices to prove the following claim.

\begin{claim}
$\Lambda$ is discrete with respect to the compact open topology.
\end{claim}

For the proof of this claim, the condition ${\rm Sing}(\phi)\subset {\rm Int}(F)$ is assumed in \cite{CSW}.
We prove this claim without the assumption.
Note that if ${\rm Sing}(\phi)\not\subset {\rm Int}(F)$, the singular Sol metric is not necessarily complete.
Let $\Lambda '<\Lambda$ be the subgroup consisting of isometries that preserve each fiber of $\widetilde{M}$ set-wise.
We first prove that the subgroup $\Lambda'$  is discrete.
Let $S$ be a fiber of $\widetilde{M}$ and $\bar{S}$ be its completion with respect to the singular Euclidean metric. 
Then we will extend $p=\pi\vert_{S}:S\map {\rm Int}(F)$ to a local isometry $\bar{p}:\bar{S} \map {\rm Int}(F) \cup {\rm Sing}(\phi)$.
Let $\{x_i\}$ be a Cauchy sequence in $S$. 
Then $\{p(x_i)\}$ is a Cauchy sequence in ${\rm Int}(F)$ and it converges to either an interior point of $F$ or a point in ${\rm Sing}(\phi)$.
Since $\bar{S}$ consists of the equivalence classes of Cauchy sequences in $S$, we can define $\bar{p}: \displaystyle [(x_i)]\mapsto \lim p(x_i)$.
Since $p$ is a local isometry, $\bar{p}$ is well defined and a local isometry.
Therefore we get $E := \bar{S}\setminus S = \bar{p}^{-1}({\rm Sing}(\phi))$ for the natural extension $\bar{p}$ of $p$. 
Any isometry $\varphi:S\map S$ extends to an isometry $\bar{\varphi}: \bar{S} \map \bar{S}$ 
and by construction we get $\bar{\varphi}(E) = E$.
Suppose there is a sequence $\{\bar{\varphi}_i\}$ of isometries such that $\bar{\varphi}_i\map {\rm id}$.
Since the distances between two distinct points in $E$ are bounded from below by a positive constant,
for large enough $i$, $\bar{\varphi}_i$ must fix $E$ point-wise.
Suppose that $\bar{\varphi}: \bar{S}\map \bar{S}$ is an isometry which preserves $E$ point-wise.
Since $\bar{S}$ is a singular Euclidean plane, we may find two points $e_1,e_2$ in $E$ 
which can be joined by a unique geodesic $\gamma$. 
Then by appealing to the distance from $e_1$ and $e_2$, it follows that $\bar{\varphi}$ preserves $\gamma$ point-wise.
Note that every isometry on $\bar{S}$  leaves the set of leaves of $p^{-1} (\mathcal{F}^s)$ and $p^{-1} (\mathcal{F}^u)$ invariant.
This implies that every leaf that intersects with $\gamma$ is preserved by $\bar{\varphi}$.
Let $l$ be one of such leaves.
Since $\bar\varphi$ is a local isometry of Sol metric, it locally acts as a translation on $\bar{S}$.
Therefore $\bar\varphi$ fixes $l$ point-wise.
Since each leaf of $\mathcal{F}^s$ or $\mathcal{F}^u$ is dense in ${\rm Int}(F)$, the orbit of $l$ under the action of 
the deck transformation group associated to $p$ is also dense in $\bar{S}$.
Hence $\bar{\varphi}$ is identity on a dense subset of $\bar{S}$ and since it is an isometry, we get $\bar{\varphi} = {\rm id}$.
Therefore for large enough $i$, we get $\bar{\varphi}_i = {\rm id}$.  
This proves the discreteness of $\Lambda'$.

The discreteness of the dynamical direction of $\Lambda$ follows from exactly the same argument in \cite{CSW}.
We include the proof for completeness.
Note that each isometry $\varphi\in \Lambda$ extends to the metric completion $\bar{M}$ of $\widetilde{M}$.
We may parametrize each fiber by real numbers $t$ in such a way that 
for any two fixed flow lines $a(t),b(t) \in E(t)\subset \bar{S}(t)$,
the distance between $a(t)$ and $b(t)$ is $\sqrt{e^{2t}x^2+e^{-2t}y^2}$ for some fixed $x$ and $y$ when $|t|$ is small enough.
Then for small $|t|$ the distance between any two points in $E(t)$ are bounded from below by a constant which does not 
depend on $t$.
Therefore since $\sqrt{e^{2t}x^2+e^{-2t}y^2}$ is not a locally constant function,
an isometry $\varphi\in \Lambda$ close enough to the identity must fix each fiber of 
the foliation by the singular Euclidean planes.
Thus we see that $\Lambda$ is discrete.

Since isometries may fix only singular points, if ${\rm Int}(F)\cap {\rm Sing}(\phi) = \emptyset$, 
then $\Lambda$ has no fixed point in $\widetilde{M}$ and the last assertion holds.
\end{proof}

\begin{cor}\label{HKT}
All the fibrations of $M_1 = S^3\setminus 6^2_2$ and the Magic 3-manifold $M_2$ are the minimal elements.
\end{cor}
\begin{proof}
Since $S^3\setminus 6^2_2$ (resp. the Magic 3-manifold) is homeomorphic to 
the complement of the fibered link associated to $\sigma_1\sigma_2^{-1} \in B_3$ (resp. $\sigma_1\sigma_2^{-1}\sigma_1 \in B_3$),
where $B_3$ is the braid group on $3$ strands (see Figure \ref{fclosure}).
It is well known that for every pseudo-Anosov element of $B_3$, all singularities are punctured.
Therefore it suffices to prove that $M_1$ and $M_2$ are minimal {\em manifolds} (not orbifolds) with respect to usual covering relation.
Note that ${\rm Vol}(S^3\setminus 6^2_2) = 4V_0$  where $V_0 \approx 1.01...$ is the volume of ideal regular tetrahedron (see for example \cite{GMM}).
Then by a result of Cao-Meyerhoff \cite{CM}, it can only cover the figure eight knot complements or its sister
(m$004$ or m$003$ in SnapPea notation).
However SnapPy \cite{SnapPy}, can enumerate all double covers of m$003$ and m$004$ and none of them are homeomorphic to $S^3\setminus 6^2_2$.
Similarly, the Magic 3-manifold $M_2$ has volume $\approx 5.33...$ and if it covers a manifold by degree $2$, then its volume is $\approx 2.66... < V_8$ where
$V_8 \approx 3.66...$ is the volume of ideal regular octahedron.
Then by a result of Agol \cite{Ago1}, it has only one cusp and can not be doubly covered by $M_2$ a 3-cusped manifold.
Moreover, since ${\rm Vol}(M_2)/3 \approx 1.77... < 2V_0$, again by Cao-Meyerhoff, $M_2$ cannot cover any manifold by degree greater than 2.
Then the result follows from the last assertion of Theorem \ref{Hyperbolic}.
\end{proof}

\begin{figure}[h]
\includegraphics[bb = 0 0 532 436, scale=0.3]{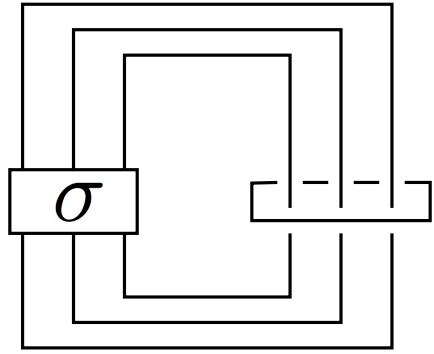}
\caption{The fibered link associated to a braid $\sigma\in B_3$}
\label{fclosure}
\end{figure}

\begin{rmk}
For a fixed surface, there exists a pseudo-Anosov automorphism with the smallest dilatation (c.f. \cite{Iva}).
It is interesting to compute the smallest dilatation for a given surface.
In \cite{H} and \cite{KT}, Hironaka and Kin-Takasawa computed dilatations of the monodromy of each fiber of $S^3\setminus 6^2_2$ and the Magic 3-manifold respectively.
It turns out that many small dilatation pseudo-Anosov automorphisms appear as the monodromies of fibrations of those manifolds.
Corollary \ref{HKT} shows that all such fibrations are minimal and hence their monodromies can be 
candidates for the smallest dilatation pseudo-Anosov maps.
\end{rmk}

\subsection{Transitivity of commensurability in Definition \ref{def.cms}}\label{ssec.trans}
In this subsection, we discuss the subtle difference between
fibered commensurability and commensurability in the sense of Definition \ref{def.cms}. 
Here, two pairs of type $(F,\phi)$ are said to be {\em fibered commensurable} if associated fibered pairs are commensurable.
It is easy to see that if two pairs $(F_1,\phi_1)$ and $(F_2,\phi_2)$ are fibered commensurable, 
they are commensurable in the sense of following definition.
\begin{defi}[\cite{Car}]\label{def.Car}
Two pairs $(F_1, \phi_1)$ and $(F_2, \phi_2)$ are {\em commensurable} if there is a surface $\widetilde{F}$,
an automorphism $\widetilde{\phi}$, and nonzero integers $k_1$ and $k_2$, so that $(\widetilde{F},\widetilde\phi)$ covers
$(F_i,\phi_i^{k_i})$ for $i = 1,2 $.
\end{defi}

In \cite{CSW}, it is claimed without proof that 
two pairs $(F_1,\phi_1)$ and $(F_2,\phi_2)$ are fibered commensurable if and  only if
they are commensurable in the sense of Definition \ref{def.cms}.
Since a map can not always be lifted even if its power can be lifted, the claim is not trivial.
The claim would follow from the transitivity of commensurability relation in the sense of Definition \ref{def.cms} 
because taking power of the automorphism is actually a covering.
In this subsection, we will prove that the transitivity of commensurability in Definition \ref{def.cms} is valid 
if the automorphisms are pseudo-Anosov.
\begin{prop}\label{prop.trans}
Suppose that $(F_i, \phi_i)$ and $(F_{i+1}, \phi_{i+1})$ are commensurable in the sense of Definition \ref{def.cms} for $i=1,2$.
Suppose further that $\phi_i$ are pseudo-Anosov for $i=1,2,3$.
Then there exists a pair $(F_{123},\widetilde\phi_i)$ that covers $(F_i,\phi_i)$ for each $i=1,2,3$ such that 
$\widetilde{\phi_1}^{k_1} = \widetilde{\phi_2}^{k_2} = \widetilde{\phi_3}^{k_3}$ for some 
$k_1,k_2,k_3\in \mathbb{Z}\setminus\{0\}$.
In particular, commensurability in the sense of Definition \ref{def.cms} is transitive.
\end{prop}
\begin{proof}
In Theorem \ref{Hyperbolic} we proved that each hyperbolic fibered commensurability class contains a unique minimal element.
Let $M = [F_1,\phi_1]$.
Recall that $\Lambda$ is the group of isometries of the singular Sol metric on the universal cover $\widetilde{M}$ (see the proof of Theorem \ref{Hyperbolic}).
Then by considering the subgroup $\Lambda^+$ that consists of isometries which preserve the orientation of $\widetilde{M}$ 
and the orientation of the leaf space of $\widetilde{M}$.
By taking $M_{\min}^+:=\widetilde{M}/\Lambda^+$, we get a unique minimal element among all commensurable fibered pairs both orientable and co-orientable.
Although there is a natural extension of this proof in the case where $\widetilde{M}/\Gamma^+$ is an orbifold,
such a proof would require more terminology and could obfuscate the key ideas of the proof.
Therefore, we only present the case where $\widetilde{M}/\Gamma^+$ is a manifold.
In this case we get an associated pair $(F_{\min},\phi_{\min})$ since $M_{\min}^+$ is orientable and co-orientable.
Each $(F_i,\phi_i)$ covers $(F_{\min},\phi_{\min}^{l_i})$ for some $l_i\in \mathbb{Z}\setminus\{0\}$ $(i=1,2,3)$.
Note that $\phi_{\min}$ is not always lifted to $F_i$.
Let $H_i<\pi_1(F_{\min})$ be a subgroup which is the image of $\pi_1(F_i)$ by the covering map for each $i=1,2,3$.
Further let $d = [\pi_1(F_{\min}):H_1\cap H_2\cap H_3]$.
Then we take $H_{123}:=\bigcap\{H<\pi_1(F_{\min})\mid [\pi_1(F_{\min}):H] = d\}$.
Recall that for a group $G$, a subgroup $H<G$ is called {\em characteristic} if for every isomorphism $f:G\map G$, 
we get $f(H) = H$.
$H_{123}$ is a characteristic subgroup and hence every homeomorphism on $F_{\min}$ lifts to the covering $F_{123}$
that corresponds to $H_{123}<\pi_1(F_{\min})$.
Since each $\phi_i:F_i\map F_i$ is a lift of $\phi_{\min}^{l_i}$, it can be lifted to $\widetilde\phi_i:F_{123}\map F_{123}$.
Let $l$ be the least common multiple of $l_i$'s, then by putting $k_i = l/l_i$, we get 
$\widetilde{\phi_1}^{k_1} = \widetilde{\phi_2}^{k_2} = \widetilde{\phi_3}^{k_3}$ on $F_{123}$.
\end{proof}
\begin{rmk}
We do not know if the transitivity, or the equivalence of fibered commensurability and commensurability in the sense of Definition \ref{def.cms} holds for the case where the automorphisms are periodic or reducible.
\end{rmk}

\section{Thurston norm and normalized entropy}\label{sec.entropy}
\subsection{Thurston norm}
Let $M$ be a fibered hyperbolic 3-manifold.
In this subsection we recall briefly the Thurston norm on $H^1(M;\mathbb{R})$ and 
discuss the relationship between fibered commensurability of fibrations on a fixed manifold $M$  and the normalized entropy.
For more details about the Thurston norm, see \cite{Thu.norm}, \cite{Kapo} and \cite{KT}.
For any (possibly disconnected) compact surface $F = F_1\sqcup F_2 \sqcup \cdots F_n$, 
let $\chi_-(F)$ be the sum of the absolute values of Euler characteristics $|\chi(F_i)|$ of components with negative Euler characteristics.
For a given $\omega \in H^1(M;\mathbb{Z})\subset H^1(M;\mathbb{R})$, we define $\|\omega\|$ to be

\begin{tabbing}
$\min\{\chi_-(F)\mid$\= $F$ is an embedded orientable surface $(F,\partial F)\subset (M,\partial M)$, and \\
			        \> $[F]\in H_2(M,\partial M;\mathbb{Z})$ is the Poincare dual of $\omega\in H^1(M;\mathbb{Z})\}.$
\end{tabbing}
If $F$ realizes the minimum, we call $F$ a minimal representative of $\omega$.
We can extend this norm to $H^1(M;\mathbb{Q})$ by $\|\omega\| = \|r\omega\|/r$.
It turns out that $\|\cdot\|$ extends continuously to $H^1(M;\mathbb{R})$.
Further, this $\|\cdot\|$ turns out to be semi-norm on $H^1(M;\mathbb{R})$ and the unit ball 
$U = \{\omega\in H^1(M;\mathbb{R})\mid \|\omega\| \leq 1\}$ is a compact convex polygon \cite{Thu.norm}.
The semi-norm $\|\cdot\|$ is called the Thurston norm on $H^1(M;\mathbb{R})$.
We need some more terminologies to explain the relationship between $\|\cdot\|$ and fibrations on $M$.
We denote
\begin{itemize}
\item the cone over a top dimensional face $\Delta$ of the unit ball $U$ by $C_{\Delta}$,
\item the set of integral classes on ${\rm Int}(C_{\Delta})$ by ${\rm Int}(C_{\Delta}(\mathbb{Z}))$, and
\item the set of rational classes on a top dimensional face $\Delta$ by $\Delta(\mathbb{Q})$.
\end{itemize}
Then Thurston proved
\begin{thm}[\cite{Thu.norm}]\label{thm.norm}
Let $M$ be a fibered hyperbolic 3-manifold and $F$ the fiber.
Then there is a top dimensional face $\Delta$ of $U$ such that
\begin{itemize}
\item the dual of $[F]\in H_2(M,\partial M;\mathbb{Z})$ belongs to ${\rm Int}(C_{\Delta}(\mathbb{Z}))$, and
\item for every primitive class $\omega$ in ${\rm Int}(C_{\Delta}(\mathbb{Z}))$, 
a minimal representative of $\omega$ is the fiber of a fibration on $M$.
\end{itemize}
\end{thm}
We call the face $\Delta$ in Theorem \ref{thm.norm} a {\em fibered face} and the cone over a fibered face a {\em fibered cone}.

As a corollary, we see that if the first Betti number $b_1(M)>1$ and $M$ is fibered, then $M$ has infinitely many distinct fibrations.
We will discuss fibered commensurability of fibrations of a hyperbolic fibered 3-manifold.

\subsection{Normalized entropy}
The normalized entropy is shared by commensurable fibrations on a fixed hyperbolic 3-manifold.
\begin{prop}\label{prop.normalized_ent}
Suppose that $[F_1,\phi_1] = [F_2,\phi_2]$ and their interior admit hyperbolic metrics.
If $(F_1,\phi_1)$ is commensurable to $(F_2,\phi_2)$, then $$\chi(F_1)\log(\lambda(\phi_1)) = \chi(F_2)\log(\lambda(\phi_2)).$$
\end{prop}
\begin{proof}
There are pairs $(\widetilde{F},\widetilde{\phi_i})$ that cover $(F_i,\phi_i)$ and $k_i\in\mathbb{Z}\setminus\{0\}$
for $i = 1,2$ such that $\widetilde{\phi_1}^{k_1} = \widetilde{\phi_2}^{k_2}$.
Then the mapping torus $[\widetilde{F},\widetilde{\phi_i}^{k_i}]$ covers $[F_i,\phi_i]$ and the degree of
this cover is $k_i\chi(\widetilde{F})/\chi(F_i)$.
Since $[F_1,\phi_1] = [F_2,\phi_2]$, we get 
$k_1/\chi(F_1) = k_2/\chi(F_2)$.
Since $\lambda(\phi) = \lambda(\widetilde\phi)$,
$$\chi(\widetilde{F})\log(\lambda(\widetilde{\phi_i}^{k_i})) = \frac{\chi(\widetilde{F})}{\chi(F_1)}\chi(F_1)k_1\log(\lambda(\phi_1))
= \frac{\chi(\widetilde{F})}{\chi(F_2)}\chi(F_2)k_2\log(\lambda(\phi_2)).$$
Putting them all together, we get $\chi(F_1)\log(\lambda(\phi_1)) = \chi(F_2)\log(\lambda(\phi_2)).$
\end{proof}
Each primitive integral class in $C_\Delta(\mathbb{Z})$ corresponds to a rational class in ${\rm Int}(\Delta)$.
The normalized entropy defines a function ${\rm ent}:\Delta(\mathbb{Q})\map \mathbb{R}$.
In \cite{Fri}, the function $\frac{1}{\rm ent}$ is shown to be concave and therefore it extends to ${\rm Int}(\Delta)$.
Moreover, 
\begin{thm}[\cite{McM}]\label{thm.concave}
$\frac{1}{\rm ent}:{\rm Int}(\Delta)\map \mathbb{R}$ is {\em strictly} concave.
\end{thm}

In Example 3.12 of \cite{CSW}, it is remarked that some fibrations on $S^3\setminus 6^2_2$ are not commensurable.
In Corollary \ref{HKT}, it is proved that all fibrations on $S^3\setminus 6^2_2$ are minimal elements 
and since each minimal element is unique, we see that two fibrations of $S^3\setminus 6^2_2$ are either symmetric or non-commensurable.
Here, we give an alternative proof of this fact in terms of the normalized entropy.
In \cite{H} or \cite{McM},  the unit ball of the Thurston norm on $H^1(S^3\setminus 6^2_2)$ is computed to be a square.
Further, the symmetries of the square all come from the symmetries of the manifold 
(see Example \ref{Hironaka} for more details about the symmetries of $S^3\setminus 6^2_2$).
Therefore the function $\frac{1}{\rm ent}$ is invariant under the action of the symmetries of the unit ball. 
Since $\frac{1}{\rm ent}$ is {\em strictly} concave, this proves that any two fibrations that correspond to
distinct elements in $H^1(M;\mathbb{Z})$
are either symmetric or non-commensurable.
In other words, the normalized entropy determines the commensurability class of a fibration on $S^3\setminus 6^2_2$ up to
symmetry.

On the other hand, in \cite[\S 2]{KKT}, it is observed that for the Magic 3-manifold $N$ there are rational points on a
fibered face which share the same normalized entropy but which are not symmetric to each other.
However again by Corollary \ref{HKT}, we also see that any two distinct fibrations of $N$ are either symmetric or non-commensurable.
Hence for the Magic 3-manifold, the commensurability classes of fibrations are not determined by the normalized entropies.
We do not know for what kind of hyperbolic 3-manifolds,
the commensurability classes of fibrations on the same hyperbolic 3-manifold
are determined by the normalized entropy up to symmetry.

\section{Commensurability of fibrations on a hyperbolic 3-manifold}\label{cms.mfd}
In this section we prove Theorem \ref{thm.hidden} and Theorem \ref{thm.genus}.
\subsection{Manifolds without hidden symmetries}
In this subsection we prove Theorem \ref{thm.hidden}.
First we prepare some definitions.
A {\em Kleinian group} is a discrete subgroup of $\PSL$.
Two Kleinian groups $\Gamma_1$ and $\Gamma_2$ are said to be {\em commensurable}
if $\Gamma_1\cap\Gamma_2$ is a finite index subgroup of both $\Gamma_1$ and $\Gamma_2$.
Let $\Gamma$ be a Kleinian group. Then the {\em commensurator} $C^+ (\Gamma)$ of $\Gamma$ is 
$$C^+ (\Gamma) = \{h \in \mathrm{PSL}(2,\mathbb{C})~\vert ~ \Gamma ~{\rm and}~  h\Gamma h^{-1}~ {\rm are~commensurable} \},$$
and the {\em normalizer} $N^+ (\Gamma)$ is $$N^+ (\Gamma) = \{h \in \mathrm{PSL}(2,\mathbb{C})~\vert ~ \Gamma = h\Gamma h^{-1}\},$$
Note that $N^+(\Gamma) < C^+(\Gamma)$.

Let $M$ be a hyperbolic 3-manifold and $\rho :\pi_1(M) \map \Gamma < \mathrm{PSL}(2,\mathbb{C})$ a holonomy representation of $\pi_1(M)$.
By the Mostow-Prasad rigidity theorem, any self-homeomorphism $\varphi:M\map M$ corresponds to a conjugation of $\Gamma$.
Therefore we get $N(\Gamma)/\Gamma \cong {\rm Isom}(M)$ where ${\rm Isom}(M)$ is the group of self-homeomorphisms of $M$.
If $C^+(\Gamma)\setminus N^+(\Gamma) \not= \emptyset$, each non-trivial element $h\in C^+(\Gamma)\setminus N^+(\Gamma)$ is said to be a {\em hidden symmetry}.
Then $M$ is said to have no hidden symmetries if $\Gamma$ has no hidden symmetries.
Note that by the Mostow-Prasad rigidity theorem, the holonomy representations of $\pi_1(M)$ are related by a conjugation.
Hence the definition does not depend on the choice of a holonomy representation.

\begin{proof}[Proof of Theorem \ref{thm.hidden}]
Let $(M,\mathcal{F}_1)$ and $(M,\mathcal{F}_2)$ be commensurable fibered pairs
that correspond to two distinct fibrations on $M$.
Then by Theorem \ref{Hyperbolic} we have a unique minimal element $(N,\mathcal{G})$ in the commensurability class.
Let  $\rho:\pi_1(N) \map \mathrm{PSL}(2,\mathbb{C})$ be a holonomy representation and $\Gamma := \rho(\pi_1(N))$.
Since $(M,\mathcal{F}_1)$ and $(M,\mathcal{F}_2)$ cover $(N,\mathcal{G})$, there are two corresponding coverings $p_1,p_2:M\map N$.
Let $\Gamma_i = \rho p_{i\ast}(\pi_1(M))$ for $i=1,2$.
Then by the Mostow-Prasad rigidity theorem, there is $h\in \mathrm{PSL}(2,\mathbb{C})$ such that $h\Gamma_1 h^{-1} = \Gamma_2$.
Further, since $\Gamma_2 <\Gamma \cap h\Gamma h^{-1}$, $h\in C^+(\Gamma) = C^+(\Gamma_1) = N^+(\Gamma_1)$.
The last equality holds since $M$ has no hidden symmetries.
Then it follows that $\Gamma_1 = \Gamma_2$ and 
hence there exists a homeomorphism $\varphi:M\map M$ such that $p_1\varphi = p_2$.
Therefore $\omega_1$ and $\omega_2$ are symmetric.
\end{proof}
\begin{rmk}
Hyperbolic 3-manifolds with hidden symmetries are ``rare" among all non-arithmetic 
hyperbolic 3-manifolds (see for example, \cite{GHH}).
Hence we may expect that ``most" hyperbolic 3-manifolds have no hidden symmetries and therefore
have no non-symmetric but commensurable fibration.
\end{rmk}
\begin{rmk}
As mentioned above,
there are no non-symmetric but commensurable fibrations on $S^3\setminus6_2^2$ and the Magic 3-manifold.
However, $S^3\setminus6_2^2$ and the Magic 3-manifold are arithmetic and by a result of Margulis \cite{Mar},
they have lots of hidden symmetries. 
Therefore even though a manifold has hidden symmetries, it might not have any non-symmetric but commensurable fibrations.
\end{rmk}

\subsection{Non-symmetric and commensurable fibrations}
In this subsection, we prove Theorem \ref{thm.genus} by constructing examples of
manifolds that have non-symmetric but commensurable fibrations.
\begin{lem}\label{lem.sym}
Let $M$ be a fibered hyperbolic 3-manifold.
Suppose two primitive elements $\omega_1\not= \pm\omega_2 \in H^1(M;\mathbb{Z})$ 
correspond to fibrations with the fibers and the monodromies $(F_1, \phi_1)$ and $(F_2, \phi_2)$ respectively.
We suppose further $(F_1, \phi_1) = (F_2, \phi_2)$ (i.e. conjugate to each other).
Then, for all large enough $n\in \mathbb{N}$, 
there exists a degree $n$ covering space $p_n:M_n\map M$ such that $p_n^\ast(\omega_1)$ and 
$p_n^\ast(\omega_2)$ correspond to commensurable but non-symmetric fibrations.
\end{lem}
\begin{proof}
Note that by the universal coefficient theorem, we have 
$$H^1(M;\mathbb{Z}) \cong {\rm Hom}(H_1(M)/{\rm Tor}, \mathbb{Z}),$$
where ${\rm Tor}$ is the torsion part.
This isomorphism is determined by a choice of a basis of $H_1(M;\mathbb{Z})/{\rm Tor}$.
Let $A_i = {\rm ab}(\pi_1(F_i))/{\rm Tor}$ where ${\rm ab}:\pi_1(M) \map H_1(M)$ is the abelianization and $\pi_1(F_i)\hookrightarrow \pi_1(M)$ is an injection induced by the fiber bundle structure of $M$ associated to $(F_i,\phi_i)$ for $i = 1,2$.
Then the fiber bundle structure of $M$ gives the exact sequence
$$0\map \pi_1(F_i)\map \pi_1(M)\xrightarrow{\rho_i} \pi_1(S^1)\cong \mathbb{Z}\map 0.$$
The map $\rho_i$ factors through the abelianization since $\pi_1(S^1)\cong \mathbb{Z}$ is abelian.
Hence we get $A_i = {\rm Ker}(\omega_i)\cong \mathbb{Z}^{b-1}$ where $b$ is the first Betti number of $M$.
Then we consider the dynamical covering $p_n:M_n\map M$ of degree $n$ with respect to $\omega_1$
(i.e. the covering corresponding to $(F_1,\phi_1^n)$).
This is the covering corresponding to the surjective map $\pi_1(M)\xrightarrow{\rm ab} H_1(M)\xrightarrow{\omega_1}\mathbb{Z}\map \mathbb{Z}/n\mathbb{Z}$.
Then for sufficiently large $n$, there exists $a\in A_2$ such that $a$ maps to a nonzero element by the above surjective map.
This means that each component of $p_n^{-1}(F_2)$ is not homeomorphic to $F_2$.
\end{proof}
\begin{ex}
The 3-manifold $S^3\setminus 6^2_2$ and the Magic manifold have symmetries that permute cusps,
and therefore they do have two distinct elements in their first cohomology 
with homeomorphic fibers and conjugate monodromies.
\end{ex}

\begin{ex}\label{Hironaka}
In this example we observe that $M := S^3\setminus 6^2_2$ has two symmetric fibrations in the same fibered cone in $H^1(M)$.
Although this fact can be checked by computing the symmetry group by SnapPy \cite{SnapPy},
we give a geometric proof.
The first half of the following argument is due to Eriko Hironaka, see also \cite{H}.

Let $u,t$ be the generators of $H_1(M,\mathbb{Z})$ 
that correspond to the meridians of $6^2_2$ (see the left picture of Figure \ref{fig.inv}).
Then let $U,T\in H^1(M;\mathbb{Z})$ be the dual of $u,t$ respectively.
Then $U$ corresponds to the fibration of $M$ with monodromy $f$ that corresponds to $\sigma_1\sigma_2^{-1}\in B_3$. 
Let $h$ be a $\pi$-rotation which is depicted in Figure \ref{fig.inv}.
We can see that $f^{-1} = \sigma_2\sigma_1^{-1} = hfh$, that is $f$ and $f^{-1}$ are conjugate to each other.
Then we take the mirror image of $6^2_2$.
By isotopy and above conjugacy, we see that $6^2_2$ is amphicheiral. 
The induced map on $H^1(M;\mathbb{Z})$ of the symmetry on $M$ that gives amphicheirality satisfy
 $U\mapsto -U$ and $T\mapsto T$.
This symmetry preserves the fibered face $\Delta := \{aU+bT \mid -1<a<1, b=1\}$.
By this symmetry, we see that fibrations on the cone $C_\Delta$ over $\Delta$ of the form $nU+mT$ and $-nU+mT$ 
($n,m\in\mathbb{Z}$) are symmetric.
\begin{figure}[h]
\includegraphics[bb = 0 0 1472 472, scale=0.2]{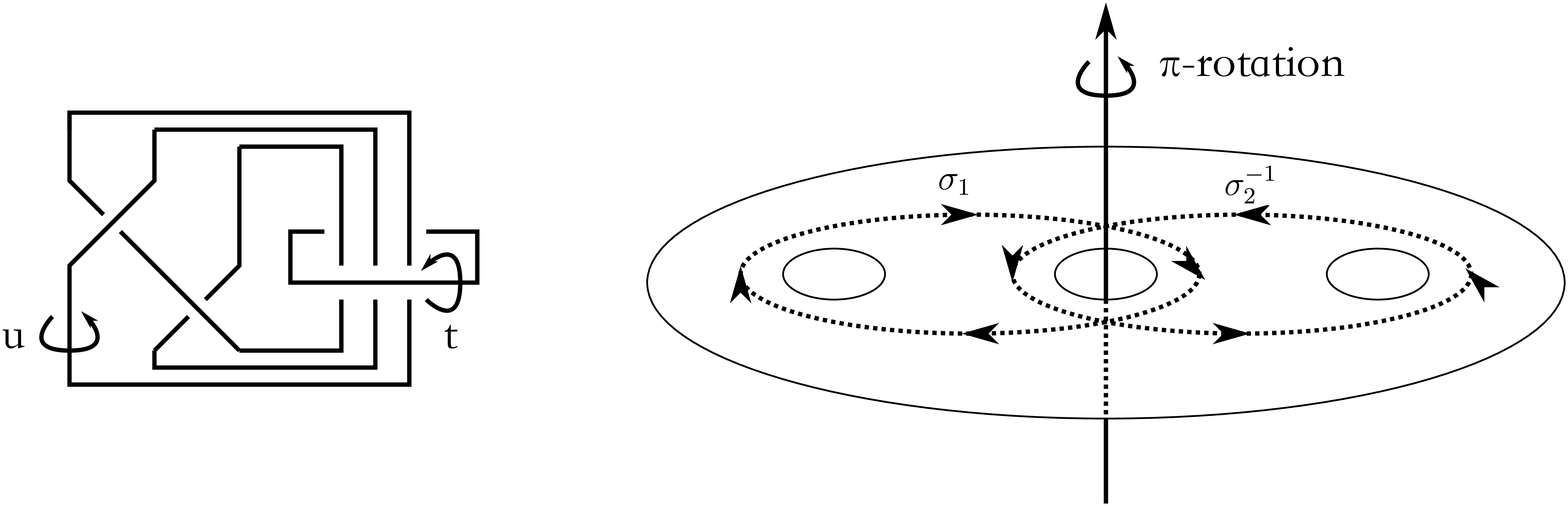}
\caption{A involution map $h$ on 4-holed sphere.}
\label{fig.inv}
\end{figure}
\end{ex}
\begin{proof}[Proof of Theorem \ref{thm.genus}]
Putting Lemma \ref{lem.sym} and Example \ref{Hironaka} together, we have a proof.
\end{proof}



\begin{thebibliography}{99}
\bibitem{Ago1}
I. Agol,
The minimal volume orientable hyperbolic 2-cusped 3-manifolds,
Proceedings of the American Mathematical Society 138 (2010), 3723-3732.
\bibitem{CSW} D. Calegari and H. Sun and S. Wang,
On fibered commensurability,
Pacific J. Math. 250 (2011), no. 2, 287-317, arXiv:1003.0411v1.
\bibitem{CM} C. Cao and G. R. Meyerhoff, 
The orientable cusped hyperbolic 3-manifolds of minimum volume.
Invent. Math. 146 (2001), no. 3, 451-478. 
\bibitem{Car} J. D. Carlson, 
Commensurability of two-multitwist pseudo-Anosovs, 
preprint, arXiv:1011.0247.
\bibitem{CB}
A. J. Casson and S. A. Bleiler,
Automorphisms of surfaces after Nielsen and Thurston, London Mathematical Society
Student Texts 9, Cambridge University Press, Cambridge, 1988.
\bibitem{SnapPy} M. Culler, N. M. Dunfield, and J. R. Weeks, 
SnapPy, 
a computer program for studying the geometry and topology of 3-manifolds. Available at http://snappy.computop.org.
\bibitem{Fri}D. Fried, 
Flow equivalence, hyperbolic systems and a new zeta function for flows, 
Commentarii Mathematici Helvetici 57 (1982), 237-259.
\bibitem{GMM} F.W. Gehring, C. Maclachlan, G.J. Martin,
Two-generator arithmetic Kleinian groups II,
Bull. London Math. Soc 30 (1998) 258-266.
\bibitem{GHH} O. Goodman, D. Heard and C. Hodgson,
Commensurators of cusped hyperbolic manifolds,
Experiment. Math. 17 (2008), 283-306.
\bibitem{H} E. Hironaka, 
Small dilatation pseudo-Anosov mapping classes coming from the simplest hyperbolic braid, 
Algebraic and Geometric Topology 10 (2010), 2041-2060.
\bibitem{Iva}
N. V. Ivanov, Coefficients of expansion of pseudo-Anosov homeomorphisms, Zap. Nauchu.
Sem. Leningrad. Otdel. Mat. Inst. Steklov. (LOMI), 167 (1988), Issled. Topol. 6, 111-116, 191,
translation in Journal of Soviet Mathematics, 52 (1990), 2819-2822.
\bibitem{Kapo} M. Kapovich,
Hyperbolic manifolds and discrete groups,
Progress in Mathematics, 183.Birkhuser Boston, Inc., Boston, MA, 2001.
\bibitem{KKT} E. Kin, S. Kojima and M. Takasawa, 
Minimal dilatations of pseudo-Anosovs generated by the magic 3-manifold and their asymptotic behavior, 
preprint, arXiv:1104.3939.
\bibitem{KT} E. Kin and M. Takasawa,
Pseudo-Anosov braids with small entropy and the magic 3-manifold,
Communications in Analysis and Geometry 19, Number 4 (2011), 705-758.
\bibitem{Mar} G. A. Margulis,
Discrete subgroups of semisimple Lie groups, Ergebnisse der Mathematik und ihrer Grenzgebiete 17,
Springer-Verlag, New York, 1991.
\bibitem{McM} C. McMullen,
Polynomial invariants for fibered 3-manifolds and Teichm\"uler geodesic for foliations,
 Annales Scientifiques de lf\'Ecole Nor male Sup\'erieure. Quatri\`eme S\'erie 33 (2000), 519-560.
\bibitem{OK} J. P. Otal and L. Kay,
The hyperbolization theorem for fibered 3-manifolds.
American Mathematical Society (2001).
\bibitem{Thu.norm}W. P. Thurston,
 A norm for the homology of 3-manifolds, Mem. Amer. Math.
Soc. 59 (1986), no. 339, i - vi and 99-130.
\bibitem{Thu} W. P. Thurston,
On the geometry and dynamics of diffeomorphisms of surfaces, 
Bulletin of the American mathematical society 19, (1988) 417-431.
\bibitem{Thur} W. P. Thurston,
The Geometry and Topology of Three-Manifolds, a.k.a. gThurstonfs notesh,
available from the MSRI.
\bibitem{W}
G. S. Walsh, Orbifolds and commensurability,
In Interactions between hyperbolic geometry, quantum topology and number theory, volume 541 of Contemp. Math., pages 221-231. Amer. Math. Soc., Providence, RI, 2011.
\end{thebibliography}
\end{document}